\documentclass[11pt]{article}

\usepackage[T1]{fontenc}
\usepackage{lmodern}
\usepackage[margin=1.02in]{geometry}
\usepackage{microtype}
\usepackage{amsmath,amssymb,amsthm,mathtools}
\usepackage{booktabs}
\usepackage{array}
\usepackage[hidelinks]{hyperref}

\allowdisplaybreaks
\numberwithin{equation}{section}

\newtheorem{theorem}{Theorem}[section]
\newtheorem{proposition}[theorem]{Proposition}
\newtheorem{corollary}[theorem]{Corollary}
\newtheorem{lemma}[theorem]{Lemma}
\theoremstyle{definition}
\newtheorem{definition}[theorem]{Definition}
\theoremstyle{remark}
\newtheorem{remark}[theorem]{Remark}

\newcommand{\Q}{\mathbf Q}
\newcommand{\Z}{\mathbf Z}
\newcommand{\F}{\mathbf F}
\newcommand{\Gal}{\operatorname {Gal}}
\newcommand{\Cl}{\operatorname {Cl}}

\newcommand{\length}{\operatorname {length}}
\newcommand{\legp}[3]{\left(\frac{#1}{#2}\right)_{\!#3}}

\title{Explicit Twisted Hilbert Class Components\\
Beyond Classical Irregularity}
\author{Peter Chocian}
\date{25 July 2026}

\begin{document}
\maketitle

\begin{abstract}
Let \(p\equiv1\pmod6\) be prime and
\(K_p=\Q(\zeta_{3p})\).  We study the reflected circular unit
\[
 \mu_p=\frac{1+z\zeta_p}{1+\bar z\zeta_p},
 \qquad
 z=-\zeta_3^2,\quad \bar z=1-z,
\]
and its character projections.  A universal Stirling polynomial
\(P_m\) gives an exact identity between the anti-spectrum of
\(\mu_p\) and the primitive divided
\(\chi_{-3}\)-twisted Stickelberger spectrum:
\[
 P_{p-j}(h)-P_{p-j}(1-h)
 =-(2h-1)(j-1)!\,b_j ,
 \qquad h=\frac z{1+z}.
\]
Thus the locally blind lines of the reflected unit are precisely
the zeros of the corresponding divided twisted Bernoulli
eigenvalues.

For every \(p<500\), we enumerate these zeros.  Exactly twelve
character lines occur.  On each line an explicit integral
idempotent product of \(\mu_p\) is a local \(p\)-th power at the
conductor primes but not a global \(p\)-th power.  Small completely
split primes provide finite Artin certificates.  The generalized
Bernoulli number has exact \(p\)-adic valuation one in every case;
the character-wise Main Conjecture therefore proves that each
radical generates the complete Hilbert-class-field component,
which has order \(p\).  Seven of the twelve lines occur at
classically regular primes.  Hence twisted degeneracy below \(500\)
is more often invisible to ordinary irregularity than aligned with
it.  The first case, \(p=67\), is worked out in full, and a
deterministic integer-arithmetic program reproduces the enumeration
and every certificate.
\end{abstract}



\section{Introduction}

Let
\[
 K_p=\Q(\zeta_{3p}),\qquad p\equiv1\pmod6.
\]
The generalized Bernoulli values attached to odd characters of
\(\Gal(K_p/\Q)\) control the corresponding \(p\)-parts of the
ideal class group.  Herbrand--Ribet and the Main Conjecture explain
the existence and size of these components.  A different problem
is to write down an explicit Kummer radical for the associated
unramified extension and to certify its nontriviality by a finite
calculation.

This paper supplies such radicals for a natural family.  Put
\[
 \eta=\zeta_p,\qquad
 z=-\zeta_3^2,\qquad
 \bar z=1-z=-\zeta_3,
\]
and consider the norm-one circular unit
\[
 \mu_p=\frac{1+z\eta}{1+\bar z\eta}.
\]
The unit arose as an orientation coordinate in a study of
cyclotomic factorisations.  Its local expansion unexpectedly
selects exactly the primitive \(\chi_{-3}\)-twisted Bernoulli
zeros.  That identity is the structural starting point here; all
results in the paper are unconditional and independent of the
existence of a solution to Fermat's equation.

For even \(j\), \(2\leq j\leq p-3\), set
\[
 k=p-j,\qquad
 \psi=\chi_{-3}\omega^j,\qquad
 \theta=\chi_{-3}\omega^k,
\]
where \(\omega\) is the mod-\(p\) cyclotomic character.  The
characters are reflected:
\[
 \psi=\omega\theta^{-1}.
\]
The projected radical lies on the even \(\theta\)-line, while its
unramified Kummer character lies on the odd \(\psi\)-line.

The results have three layers.

\begin{enumerate}
\item A universal projector polynomial identifies the local
      anti-spectrum of \(\mu_p\) with divided twisted Bernoulli
      eigenvalues, including the exact nonzero scalar.
\item A complete scan for \(p<500\) finds twelve vanishing lines.
      For each, an explicit split-prime residue computation proves
      that the projected unit is not a global \(p\)-th power.
\item A second-order local computation proves that every zero is
      simple.  The character-wise Main Conjecture then proves that
      every resulting class component has order exactly \(p\).
\end{enumerate}

The twelve lines include seven rows at classically regular primes:
\[
 139,\ 199,\ 241,\ 331,\ 337\text{ (twice)},\ 457.
\]
This separation is not exceptional in the tested range; it is the
majority phenomenon.

The construction is explicit.  Every radical is a product with
integer exponents, and every nontriviality witness is a modular
power at a listed rational prime.  The accompanying verification
program uses only deterministic integer arithmetic.

\section{Characters and the reflected circular unit}

Write
\[
 G_p=(\Z/3p\Z)^\times=\Gal(K_p/\Q),
\qquad
 \sigma_a(\zeta_{3p})=\zeta_{3p}^a.
\]
Let \(\chi=\chi_{-3}\) be the quadratic character of conductor
\(3\), and let \(\omega(a)\equiv a\pmod p\).  Complex conjugation
shows that \(\chi\omega^j\) is odd when \(j\) is even.

\begin{lemma}
\label{lem:global-unit}
The numerator and denominator of \(\mu_p\) are global units of
\(K_p\).  Moreover,
\[
 N_{K_p/\Q(\zeta_3)}(1+z\eta)=1
\]
and
\[
 \tau(\mu_p)=\mu_p^{-1},
\]
where \(\tau\) fixes \(\eta\) and interchanges \(z,\bar z\).
\end{lemma}

\begin{proof}
Both \(1+z\eta\) and \(1+\bar z\eta\), up to roots of unity, have
the form \(1-\zeta_{3p}^u\) with \((u,3p)=1\).  Since \(3p\) has
two distinct prime divisors,
\[
 \Phi_{3p}(1)=1,
\]
so both elements are units.  More directly,
\[
 \prod_{r=1}^{p-1}(1+z\eta^r)
 =\frac{1+z^p}{1+z}=1,
\]
because \(z^p=z\) for \(p\equiv1\pmod6\).  This is the genuine
relative norm; the omitted \(r=0\) factor is essential.  Finally,
\(\tau\) exchanges numerator and denominator.
\end{proof}

For the reflected pair \((j,k)\), define integral coefficients
\[
 c_a\in\{0,\ldots,p-1\},\qquad
 c_a\equiv
 \bigl(2(p-1)\bigr)^{-1}\theta(a)^{-1}\pmod p,
\]
and put
\begin{equation}
\label{eq:projected-unit}
 \widetilde\mu_k
 =
 \prod_{a\in G_p}\sigma_a(\mu_p)^{c_a}.
\end{equation}

\begin{proposition}[Integral projected representative]
\label{prop:integral-projector}
The class of \(\widetilde\mu_k\) in
\(K_p^\times/K_p^{\times p}\) is the \(\theta\)-projection of
\(\mu_p\).  In particular,
\[
 \sigma_b[\widetilde\mu_k]
 =\theta(b)[\widetilde\mu_k]
 \quad (b\in G_p).
\]
Changing the chosen integer lifts \(c_a\) changes
\(\widetilde\mu_k\) only by a global \(p\)-th power.
\end{proposition}

\begin{proof}
As
\[
 |G_p|=\varphi(3p)=2(p-1)
\]
is prime to \(p\), the idempotent
\[
 e_\theta=\frac1{|G_p|}
 \sum_{a\in G_p}\theta(a)^{-1}\sigma_a
\]
is defined in \(\F_p[G_p]\).  The \(c_a\) are integer lifts of its
coefficients.  Adding a multiple of \(p\) to one exponent
multiplies the product by a \(p\)-th power.
\end{proof}

\section{Universal projectors and twisted reflection}

The local calculation is governed by a family of rational
polynomials independent of \(p\).

\begin{definition}[Universal projector polynomial]
\label{def:Pm}
For \(m\geq1\), define
\[
 P_m(X)=
 \frac1{m!}\sum_{n=1}^{m}
 (-1)^{m+n}(n-1)!
 \left\{\begin{matrix}m\\n\end{matrix}\right\}X^n,
\]
where the braces are Stirling numbers of the second kind.
\end{definition}

\begin{proposition}[Generating function]
\label{prop:generating}
The formal generating series of the polynomials \(P_m\) is
\[
 \boxed{\qquad
 \sum_{m\geq1}P_m(X)Y^m
 =-\log\bigl(1-X(1-e^{-Y})\bigr).
 \qquad}
\]
\end{proposition}

\begin{proof}
Expand the logarithm as
\[
 \sum_{n\geq1}\frac{X^n}{n}(1-e^{-Y})^n.
\]
The standard exponential generating function
\[
 (e^T-1)^n
 =n!\sum_{m\geq n}
 \left\{\begin{matrix}m\\n\end{matrix}\right\}
 \frac{T^m}{m!}
\]
with \(T=-Y\) gives the stated coefficient.
\end{proof}

Choose a root of \(z^2-z+1=0\) in \(\F_p\), put
\[
 \bar z=1-z=z^{-1},\qquad
 h=\frac z{1+z},\qquad
 \bar h=1-h,
\]
and note that
\begin{equation}
\label{eq:h-relation}
 (2h-1)^2=-\frac13,\qquad
 z-\bar z=3(2h-1).
\end{equation}

For even \(j\), \(2\leq j\leq p-3\), define the divided twisted
eigenvalue
\begin{equation}
\label{eq:bj}
 b_j=
 \frac1p
 \sum_{\substack{1\leq a<3p\\(a,3p)=1}}
 a\,\chi(a)\widetilde a^{-j}
 \pmod p,
\end{equation}
where \(\widetilde a\) is the Teichm\"uller lift of \(a\bmod p\)
to \((\Z/p^2\Z)^\times\).

\begin{lemma}[Finite form of the twisted eigenvalue]
\label{lem:finite-b}
For even \(j\), \(2\leq j\leq p-3\),
\[
 \boxed{\qquad
 b_j=
 3\sum_{\substack{1\leq r<p\\r\equiv2\ ({\rm mod}\ 3)}}r^{-j}
 \pmod p.
 \qquad}
\]
\end{lemma}

\begin{proof}
Group the terms in \eqref{eq:bj} by their common residue
\(r\bmod p\).  Among \(r,r+p,r+2p\), one is divisible by \(3\)
and the other two have opposite \(\chi\)-values.  Dividing the
resulting sum by \(p\), and using
\[
 \sum_{r=1}^{p-1}r^{-j}=0
 \]
for \(2\leq j\leq p-3\), gives the displayed residue-class sum.
\end{proof}

\begin{theorem}[Reflected-unit/twisted-eigenvalue identity]
\label{thm:reflection}
Let \(j\) be even with \(2\leq j\leq p-3\), and put \(k=p-j\).
Then
\[
 \boxed{\qquad
 P_k(h)-P_k(1-h)
 =-(2h-1)(j-1)!\,b_j
 \quad\text{in }\F_p.
 \qquad}
\]
Consequently,
\[
 P_k(h)-P_k(1-h)=0
 \quad\Longleftrightarrow\quad
 b_j=0.
\]
\end{theorem}

\begin{proof}
Put
\[
 C_m=P_m(h)-P_m(1-h).
\]
Proposition~\ref{prop:generating} gives
\[
 \sum_{m\geq1}C_mY^m
 =-\log\frac{1+ze^{-Y}}{z+e^{-Y}}.
\]
Differentiating and using \(z\bar z=1\) gives
\begin{equation}
\label{eq:derivative-spectrum}
 \frac d{dY}\sum_{m\geq1}C_mY^m
 =(z-\bar z)\frac{e^{-Y}-e^{-2Y}}{1-e^{-3Y}}.
\end{equation}
The right side is the exponential generating series for the
residue-class difference selected by \(\chi_{-3}\).  Extract the
coefficient of \(Y^{p-j-1}\), replace the resulting finite residue
sum by Lemma~\ref{lem:finite-b}, and use
\[
 (p-j)!(j-1)!\equiv1\pmod p
\]
for even \(j\).  Since \(z-\bar z=3(2h-1)\), the result is
\[
 C_{p-j}=-(2h-1)(j-1)!\,b_j.
\]
The remaining scalar is nonzero because \(2h-1\neq0\) and \(j<p\).
\end{proof}

\begin{remark}
The quantity \(b_j\) is the primitive divided
\(\chi_{-3}\omega^j\)-eigenvalue of the conductor-\(3p\)
Stickelberger operator.  Theorem~\ref{thm:reflection} identifies
the complete local anti-spectrum with the twisted Bernoulli
spectrum, not merely their zero sets.
\end{remark}

There is also a useful comparison between the two orientations
that led to this unit.

\begin{proposition}[One reflected line]
\label{prop:one-line}
For
\[
 S_m=P_m(h)-P_m(1-h),
\qquad
 E_m=P_m(z)-P_m(1-z),
\]
one has
\[
 \boxed{\qquad E_m=(1+2^m)S_m\qquad(m\geq1).}
\]
\end{proposition}

\begin{proof}
Let \(\mathcal S(Y)=\sum S_mY^m\) and
\(\mathcal E(Y)=\sum E_mY^m\).  From
Proposition~\ref{prop:generating},
\[
 \mathcal S'(Y)
 =(z-\bar z)\frac{e^{-Y}}{1+e^{-Y}+e^{-2Y}},
\]
\[
 \mathcal E'(Y)
 =(z-\bar z)\frac{e^{-Y}}{1-e^{-Y}+e^{-2Y}}.
\]
Direct comparison gives
\[
 \mathcal E'(Y)=\mathcal S'(Y)+2\mathcal S'(2Y).
\]
Comparison of the coefficient of \(Y^{m-1}\) proves the result.
\end{proof}

\section{The complete catalogue below 500}

The finite formula in Lemma~\ref{lem:finite-b} makes an exhaustive
scan immediate.  The next theorem records both the enumeration and
the class-field conclusion; the latter is proved in
Sections~\ref{sec:local}--\ref{sec:main-conjecture}.

\begin{theorem}[Catalogue and explicit class components]
\label{thm:catalogue}
For primes \(p<500\), \(p\equiv1\pmod6\), the equality \(b_j=0\)
holds precisely for the following twelve lines:
\[
 \begin{array}{c@{\quad}r@{\quad}r@{\qquad}c@{\quad}r@{\quad}r}
 p&j&k&p&j&k\\ \toprule
 67&20&47&331&182&149\\
 103&10&93&337&130&207\\
 139&40&99&337&136&201\\
 199&38&161&409&348&61\\
 241&220&21&421&348&73\\
 271&182&89&457&362&95.
 \end{array}
\]
For every row, the extension
\[
 K_p(\widetilde\mu_k^{1/p})/K_p
\]
is the complete
\(\psi=\chi_{-3}\omega^j\) component of the Hilbert class field,
and
\[
 \boxed{\qquad
 \#\bigl(\Cl(K_p)\otimes\Z_p\bigr)_\psi=p.
 \qquad}
\]
Every displayed local zero and every generalized Bernoulli zero is
simple.
\end{theorem}

The two rows at \(p=337\) are distinct character components.
Seven rows occur at classically regular primes:
\[
 139,\quad199,\quad241,\quad331,\quad
 337\text{ (twice)},\quad457.
\]
The remaining rows occur at the ordinarily irregular primes
\[
 67,\quad103,\quad271,\quad409,\quad421.
\]
Ordinary regularity and \(\chi_{-3}\)-twisted nondegeneracy are
therefore independent layers.  In the range \(p<500\), most of the
twisted lines are invisible to the ordinary Bernoulli test.

\section{Second-order digits}

For every row of Theorem~\ref{thm:catalogue}, Hensel lift the
smaller root of \(z^2-z+1\) to \(\Z/p^2\Z\), put
\[
 h=\frac z{1+z},
\qquad
 S_k=P_k(h)-P_k(1-h),
\]
and define, using Teichm\"uller lifts modulo \(p^3\),
\[
 T_j=
 \sum_{\substack{1\leq a<3p\\(a,3p)=1}}
 a\,\chi(a)\widehat a^{-j}
 \pmod {p^3}.
\]

\begin{proposition}[Exact divided digits]
\label{prop:digits}
The exact divided values are
\[
\begin{array}{c@{\;}r@{\;}r@{\;}r@{\;}r@{\qquad}c@{\;}r@{\;}r@{\;}r@{\;}r}
p&j&S_k/p&T_j/p^2&B/p&p&j&S_k/p&T_j/p^2&B/p\\ \toprule
67&20&4&33&11&331&182&96&98&143\\
103&10&61&12&4&337&130&318&140&159\\
139&40&75&64&114&337&136&295&147&49\\
199&38&147&149&116&409&348&207&225&75\\
241&220&210&210&70&421&348&184&344&255\\
271&182&129&29&100&457&362&134&195&65.
\end{array}
\]
All entries are reduced modulo \(p\), and
\[
 B=B_{1,\psi^{-1}}=\frac{T_j}{3p}.
\]
In particular,
\[
 v_p(S_k)=1,\qquad
 v_p(T_j)=2,\qquad
 v_p(B_{1,\psi^{-1}})=1.
\]
\end{proposition}

\begin{proof}
The \(S_k\)-column is the direct evaluation of the finite
Stirling polynomial modulo \(p^2\).  The \(T_j\)-column is the
finite sum above in \(\Z/p^3\Z\).  Every displayed divided digit is
nonzero.  Division by \(3p\) lowers the valuation by one and
multiplies the last digit by \(3^{-1}\pmod p\).
\end{proof}

\section{The local Kummer criterion}
\label{sec:local}

We now prove that a vanishing first coefficient makes the projected
unit a local \(p\)-th power at the conductor primes.  This is the
step from the reflected spectrum to an unramified extension.

\begin{theorem}[One-coordinate local criterion]
\label{thm:local-criterion}
Let
\[
 F=\Q_p(\zeta_p),\qquad
 \pi=1-\zeta_p,\qquad
 \Delta=\Gal(F/\Q_p),
\]
and let \(U_1=1+\pi\mathcal O_F\).  Fix
\(2\leq m\leq p-2\).  Suppose that a unit \(u\in F^\times\),
modulo \(F^{\times p}\), lies entirely on the
\(\omega^m\)-line.  Remove its residue-field Teichm\"uller factor
and put
\[
 \ell_m(u)=
 [\pi^m]\!
 \left(
 \sum_{r=1}^{p-1}r^{-m}\sigma_r
 \right)\log u
 \quad\text{in }\F_p.
\]
Then
\[
 \boxed{\qquad
 u\in F^{\times p}
 \quad\Longleftrightarrow\quad
 \ell_m(u)=0.
 \qquad}
\]
\end{theorem}

\begin{proof}
Let \(e_m\in\Z_p[\Delta]\) be the Teichm\"uller idempotent for
\(\omega^m\).  For \(m\neq1\), the \(p\)-adic logarithm induces a
\(\Delta\)-equivariant isomorphism
\[
 e_mU_1\xrightarrow{\ \log\ }
 M_m:=e_m\pi^2\mathcal O_F;
\]
see, for example, the standard logarithmic and ramification
filtrations in \cite{Serre}.  The additive module \(M_m\) is free
of rank one over \(\Z_p\).

Since
\[
 \sigma_r(\pi)=r\pi+O(\pi^2),
\]
the graded line
\(\pi^n\mathcal O_F/\pi^{n+1}\mathcal O_F\) has character
\(\omega^n\).  The first graded piece of \(M_m\) occurs at degree
\(m\), the next at degree \(m+p-1\), and
\[
 p=\text{a unit}\cdot\pi^{p-1}.
\]
Thus multiplication by \(p\) moves the first piece to the second,
and the leading coefficient gives an isomorphism
\[
 M_m/pM_m
 \xrightarrow{\ \sim\ }
 \pi^m\mathcal O_F/\pi^{m+1}\mathcal O_F
 \simeq\F_p.
\]
Modulo \(p\), the operator defining \(\ell_m\) is a nonzero scalar
multiple of \(e_m\).  Therefore \(\ell_m(u)=0\) if and only if
\(\log u\in pM_m\), which is equivalent to the principal-unit part
of \(u\) being a \(p\)-th power.  Every residue-field Teichm\"uller
unit is also a \(p\)-th power, because \(t^p=t\).
\end{proof}

\begin{lemma}[Restriction of the global projector]
\label{lem:restriction}
Fix one of the two primes of \(K_p\) above \(p\), equivalently one
of the two \(p\)-adic values of \(z\).  In the local Kummer group,
the product \(\widetilde\mu_k\) of
\eqref{eq:projected-unit} represents exactly the
\(\omega^k\)-projection of \(\mu_p\).
\end{lemma}

\begin{proof}
For \(r\in(\Z/p\Z)^\times\), let \(a_+\) and \(a_-\) be the two
lifts modulo \(3p\) satisfying
\[
 a_+\equiv1\pmod3,\qquad
 a_-\equiv-1\pmod3.
\]
At the selected place,
\[
 \sigma_{a_+}(\mu_p)=\mu_r,\qquad
 \sigma_{a_-}(\mu_p)=\mu_r^{-1}.
\]
Their idempotent coefficients have opposite \(\chi\)-signs.
Because the second base is inverted, the two contributions add
and give
\[
 \frac1{p-1}r^{-k}\pmod p,
\]
which is the coefficient of the local \(\omega^k\)-idempotent.
\end{proof}

\begin{proposition}[Every catalogue radical is unramified]
\label{prop:unramified}
For every row of Theorem~\ref{thm:catalogue},
\(\widetilde\mu_k\) is a local \(p\)-th power at both primes above
\(p\).  Consequently
\[
 K_p(\widetilde\mu_k^{1/p})/K_p
\]
is everywhere unramified, although this conclusion alone does not
prove that the extension is nontrivial.
\end{proposition}

\begin{proof}
At a chosen \(p\)-adic place,
\[
 \mu_p
 =z\,\frac{1-h\pi}{1-(1-h)\pi}.
\]
The constant \(z\) is a \(p\)-th power because \(z^p=z\).
The coefficient formula in Definition~\ref{def:Pm} gives
\[
 [\pi^k]
 \left(\sum_{r=1}^{p-1}r^{-k}\sigma_r\right)\log\mu_p
 =P_k(h)-P_k(1-h)=S_k.
\]
By Theorem~\ref{thm:reflection}, \(b_j=0\) makes this coefficient
zero modulo \(p\).  Lemma~\ref{lem:restriction} and
Theorem~\ref{thm:local-criterion} show that
\(\widetilde\mu_k\) is a local \(p\)-th power.  The involution
\(\tau\) exchanges the two primes and inverts the radical class,
so the same holds at the conjugate place.

Away from \(p\), a Kummer extension defined by a global unit is
unramified.  Lemma~\ref{lem:global-unit} supplies the unit
condition.
\end{proof}

\section{Finite Artin certificates}

To decide the remaining unit--class fork, let
\[
 q\equiv1\pmod {3p}
\]
be prime, choose \(\xi\in\F_q^\times\) of exact order \(3p\), and
let
\[
 \mathfrak q=(q,\zeta_{3p}-\xi),\qquad
 \eta_q=\xi^3,\qquad \rho_q=\xi^p.
\]
Reduction of the integral projected unit is the finite product
\begin{equation}
\label{eq:finite-product}
 M_{p,j}=
 \prod_{a\in G_p}
 \left(
 \frac{1+(-\rho_q^{2a})\eta_q^a}
      {1+(1+\rho_q^{2a})\eta_q^a}
 \right)^{c_a}
 \in\F_q^\times.
\end{equation}

\begin{proposition}[Complete split-prime certificate table]
\label{prop:witnesses}
For the twelve catalogue rows, the following data satisfy
\[
 M_{p,j}^{(q-1)/p}=\eta_q^e\neq1:
\]
\[
\begin{array}{c@{\;}r@{\;}r@{\;}r@{\;}r@{\qquad}c@{\;}r@{\;}r@{\;}r@{\;}r}
(p,j)&q&\xi&M&e&(p,j)&q&\xi&M&e\\ \toprule
(67,20)&1609&1363&1141&24&(331,182)&1987&4&1305&18\\
(103,10)&619&4&57&88&(337,130)&6067&64&3005&94\\
(139,40)&1669&16&1369&43&(337,136)&6067&64&339&306\\
(199,38)&2389&16&1451&40&(409,348)&4909&81&3913&39\\
(241,220)&1447&4&932&186&(421,348)&17683&15979&12249&6\\
(271,182)&1627&9&529&136&(457,362)&13711&1024&8316&147.
\end{array}
\]
Consequently no \(\widetilde\mu_k\) in the table is a global
\(p\)-th power.
\end{proposition}

\begin{proof}
For every row, \(q\) is prime, \(q\equiv1\pmod {3p}\), and the
listed \(\xi\) has exact order \(3p\).  Substitution into
\eqref{eq:finite-product} gives \(M\).  The final modular power
gives \(\eta_q^e\), with \(e\not\equiv0\pmod p\).  A global
\(p\)-th power has trivial \(p\)-th-power residue symbol at every
prime away from \(p\).
\end{proof}

\begin{corollary}
\label{cor:nontrivial-unramified}
Every row of Theorem~\ref{thm:catalogue} yields a nontrivial
everywhere-unramified cyclic extension of degree \(p\).  Its class
character is
\[
 \omega\theta^{-1}
 =\chi_{-3}\omega^j=\psi.
\]
\end{corollary}

\begin{proof}
Combine Propositions~\ref{prop:unramified} and
\ref{prop:witnesses}.  The character identity is Kummer duality,
using \(j+k=p\), hence \(1-k\equiv j\pmod {p-1}\).
\end{proof}

\section{Exact class-component orders}
\label{sec:main-conjecture}

For every catalogue row, the character
\[
 \psi=\chi_{-3}\omega^j
\]
is odd, primitive of conductor \(3p\), and distinct from
\(\omega\).  Moreover,
\[
 p\nmid[K_p:\Q]=2(p-1),
\]
so the \(p\)-adic group algebra is semisimple.  In this
semi-simple setting the algebraic and arithmetic isotypic
components of the class group coincide; the distinction between
the two notions, which becomes essential when \(p\) divides the
degree, is analyzed in \cite{GrasPhi}.

\begin{theorem}[Character-wise completeness]
\label{thm:completeness}
For every row of Theorem~\ref{thm:catalogue},
\[
 \length_{\Z_p}
 \bigl(\Cl(K_p)\otimes\Z_p\bigr)_\psi
 =
 v_p(B_{1,\psi^{-1}})=1.
\]
Thus
\[
 \#\bigl(\Cl(K_p)\otimes\Z_p\bigr)_\psi=p,
\]
and the extension from
Corollary~\ref{cor:nontrivial-unramified} is the complete
\(\psi\)-component of the Hilbert class field.
\end{theorem}

\begin{proof}
The values of \(\psi\) lie in \(\Z_p^\times\), so its coefficient
ring is \(\Z_p\).  Moreover, the \(p\)-Sylow subgroup of
\(\Gal(K_p/\Q)\) is trivial, so the faithfulness condition in
Greither's finite-level odd-character formula is automatic.
Greither's Theorem~4.1, which records the odd-\(p\),
prime-to-the-group-order case as the Mazur--Wiles case
\cite[Theorem~4.1]{Greither}, gives
\[
 \length_{\Z_p}
 \bigl(\Cl(K_p)\otimes\Z_p\bigr)_\psi
 =
 v_p(B_{1,\psi^{-1}}).
\]
This is the per-character input; the analytic
Iwasawa--Leopoldt class-number formula by itself controls only the
product over characters.  Proposition~\ref{prop:digits} makes the
right side equal to one.  The nontrivial unramified quotient
already has degree \(p\), so it exhausts the component.
\end{proof}

This completes the proof of Theorem~\ref{thm:catalogue}.

\section{The first component in full: p = 67}

We record the first row in enough detail to make the construction
independently checkable without reading the general tables.
Put
\[
 K=\Q(\zeta_{201}),\qquad
 \eta=\zeta_{67},\qquad
 \epsilon=\zeta_3,\qquad
 z=-\epsilon^2.
\]
Here
\[
 (p,j,k)=(67,20,47),\qquad
 \theta=\chi_{-3}\omega^{47}.
\]
Since \(132^{-1}=33\) in \(\F_{67}\),
\[
 c_a\equiv33\,\chi_{-3}(a)a^{-47}\pmod {67},
\]
and
\[
 \widetilde\mu_{47}
 =
 \prod_{a\in(\Z/201\Z)^\times}
 \sigma_a(\mu_{67})^{c_a}.
\]

\begin{proposition}[The local \(67\)-adic certificate]
\label{prop:p67-local}
Choose \(z\equiv30\pmod {67}\).  Modulo \(67^2\),
\[
 z=700,\qquad h=\frac z{1+z}=1730.
\]
Then
\[
 P_{47}(1730)=134,\qquad
 P_{47}(1-1730)=4355,
\]
and hence
\[
 \boxed{\qquad
 S_{47}=67\cdot4\pmod {67^2}.
 \qquad}
\]
For Teichm\"uller lifts modulo \(67^3\),
\[
 \boxed{\qquad
 \frac1{67}
 \sum_{\substack{1\leq a<201\\(a,201)=1}}
 a\,\chi_{-3}(a)\widehat a^{-20}
 =67\cdot33\pmod {67^2}.
 \qquad}
\]
Thus the local and Bernoulli zeros are both simple, and
\[
 \frac{B_{1,\chi_{-3}\omega^{-20}}}{67}
 \equiv11\pmod {67}.
\]
\end{proposition}

\begin{proof}
One checks
\[
 700^2-700+1=0\pmod {67^2},
\qquad
 700/701=1730\pmod {67^2}.
\]
The two displayed \(P_{47}\)-values are direct evaluations of
Definition~\ref{def:Pm}; their difference is
\[
 134-4355\equiv268=67\cdot4\pmod {67^2}.
\]
The Teichm\"uller sum is a finite calculation in
\(\Z/67^3\Z\).  Division by \(3\cdot67\) gives the final digit
\(11\).
\end{proof}

\begin{theorem}[Explicit unramified \(67\)-extension]
\label{thm:p67-extension}
At
\[
 \mathfrak q_{1609}=(1609,\zeta_{201}-1363),
\]
one has, in the arithmetic-Frobenius convention,
\[
 \boxed{\qquad
 \legp{\widetilde\mu_{47}}{\mathfrak q_{1609}}{67}
 =\zeta_{67}^{24}\neq1.
 \qquad}
\]
Consequently
\[
 K(\widetilde\mu_{47}^{1/67})/K
\]
is a nontrivial everywhere-unramified cyclic extension of degree
\(67\), and it is the complete
\(\chi_{-3}\omega^{20}\)-component of the Hilbert class field.
\end{theorem}

\begin{proof}
The prime satisfies
\[
 1609-1=1608=8\cdot201=24\cdot67.
\]
The integer \(7\) is a primitive root modulo \(1609\).  Put
\[
 \xi=7^8=1363,\qquad
 \eta_0=\xi^3=1141,\qquad
 \rho_0=\xi^{67}=250.
\]
These elements have orders \(201,67,3\), respectively.  Reduction
of the 132-term product gives
\[
 \prod_{a\in(\Z/201\Z)^\times}
 \left(
 \frac{1+(-\rho_0^{2a})\eta_0^a}
      {1+(1+\rho_0^{2a})\eta_0^a}
 \right)^{c_a}
 =1141=\eta_0.
\]
Raising to \((1609-1)/67=24\) gives
\[
 1141^{24}=320=\eta_0^{24}\neq1.
\]
Proposition~\ref{prop:unramified} supplies unramifiedness and
Theorem~\ref{thm:completeness} supplies completeness.
\end{proof}

\begin{corollary}[Normalized class coordinate]
\label{cor:kappa}
The residue character
\[
 \varphi_{67}\colon
 \Cl(K)_{\chi_{-3}\omega^{20}}\longrightarrow\mu_{67},
\qquad
 [\mathfrak a]\longmapsto
 \legp{\widetilde\mu_{47}}{\mathfrak a}{67},
\]
is an isomorphism.  If
\[
 \kappa_{67}([\mathfrak a])
 =14\log_{\zeta_{67}}\varphi_{67}([\mathfrak a])
 \quad\text{in }\F_{67},
\]
then
\[
 \kappa_{67}([\mathfrak q_{1609}])=1.
\]
\end{corollary}

\begin{proof}
Theorem~\ref{thm:p67-extension} gives exponent \(24\), and
\(24^{-1}=14\) in \(\F_{67}\).  Theorem~\ref{thm:completeness}
shows that the source and target both have order \(67\).
\end{proof}

\section{Negative controls and the global nature of the character}

The class character is not annihilated by several natural
primewise conditions that arose in the construction's original
geometric setting.

\begin{proposition}[Two finite controls]
\label{prop:controls}
There are split finite-field configurations satisfying a local
Fermat identity and two secant \(67\)-th-power conditions, but
having nontrivial \(\varphi_{67}\).  Explicitly:
\[
\begin{array}{c|r|r|r|c}
q&z&\bar z&\eta&\varphi_{67}\\ \hline
173263&4908&168356&47242&\eta^{57}\\
7783927&6283819&1500109&2490097&\eta^{8}.
\end{array}
\]
For the first row, take
\[
 (a,b,c)=(4908,9400,1),
\]
with
\[
 1-\bar z\eta=79276^{67},\qquad
 z+\bar z\eta=95569^{67}.
\]
For the second, take
\[
 (a,b,c)=(6283819,3760397,1),
\]
with
\[
 1-\bar z\eta=7454094^{67},\qquad
 z+\bar z\eta=2232030^{67}.
\]
In both cases,
\[
 a^{67}+b^{67}=c^{67}
\]
in the displayed finite field, while the class character is
nontrivial.  Moreover,
\[
 7783927\equiv1\pmod {67^2}.
\]
\end{proposition}

\begin{proof}
All assertions are direct modular exponentiations.  The projected
product in the first row is \(29011\), whose
\((q-1)/67\)-th power is \(\eta^{57}\).  In the second it is
\(5414522\), whose corresponding power is \(\eta^8\).
\end{proof}

\begin{corollary}
A local Fermat identity, the two displayed secant power
conditions, and even complete splitting at the
\(67^2\)-cyclotomic level do not force the
\(\chi_{-3}\omega^{20}\) coordinate to vanish.  Any vanishing
theorem for a solution-dependent divisor must use a global
correlation among its support primes.
\end{corollary}

\begin{remark}[Relation to Fermat's equation]
The construction of the unramified extensions is unconditional.
At exponent \(67\), the classical Sophie Germain auxiliary prime
\[
 269=4\cdot67+1
\]
already excludes a first-case solution.  The nonzero \(67\)-th
powers modulo \(269\) are
\[
 \{1,82,187,268\};
\]
no two differ by \(1\), and \(67\) is not a \(67\)-th power
modulo \(269\).  Thus the class-field theorem above is substantive,
whereas any assertion about its value on an actual exponent-\(67\)
Fermat solution would be vacuous.
\end{remark}

\section{Reproducibility}

The accompanying program
\[
 \texttt{verify\_twisted\_class\_generators.py}
\]
uses only integer arithmetic.  It performs the following checks:

\begin{itemize}
\item enumerates every prime \(p<500\) with \(p\equiv1\pmod6\)
      and every relevant even character index;
\item evaluates the finite sum in Lemma~\ref{lem:finite-b};
\item independently recomputes the ordinary Bernoulli irregular
      indices and the seven-of-twelve regular-row count;
\item builds the Stirling polynomials and Hensel lifts;
\item evaluates the second-order local digits and the
      Teichm\"uller sums modulo \(p^3\);
\item verifies primality and exact root orders at all twelve split
      primes;
\item evaluates every factor in each integral idempotent product;
      and
\item checks each nonzero residue-symbol exponent.
\end{itemize}

For the \(p=67\) negative controls, the computation takes place in
\[
 \Z/67^3\Z,\quad
 \F_{1609},\quad
 \F_{173263},\quad
 \F_{7783927}.
\]
No computer-algebra package or unrecorded class-group computation
is used.  The finite calculations establish the local digits and
nontriviality; the exact class-component orders use
Theorem~\ref{thm:completeness}.

\section{Interpretation, scope, and further questions}

The general relationship among generalized Bernoulli values,
circular units, reflection, and class groups is classical.  The
present construction lies in the Thaine--Rubin--Greither orbit
\cite{Thaine,Rubin,Greither}.  Its specific contribution is the
uniform reflected radical \eqref{eq:projected-unit}, the exact
spectrum identity of Theorem~\ref{thm:reflection}, the complete
below-\(500\) catalogue, and a finite Artin witness for every
catalogue line.

All twelve observed lines land on the class side of the
unit--class fork.  This is naturally a Spiegelung dichotomy: the
radical lies on the even \(\theta\)-line and the resulting
unramified character on the odd reflected \(\psi\)-line.  A
twisted Vandiver-type theorem annihilating the relevant even class
component would force the class-side outcome uniformly.  Such a
statement belongs to the same unit--class comparison developed in
Greither's proof of the Gras conjecture
\cite[Corollary~4.15]{Greither}.

The theorem is finite-range.  It does not assert that every
twisted degeneracy for arbitrary \(p\) is simple, nor that this
circular unit always selects the class side.  It proves those facts
for every line below \(500\).  Extending the computation and
explaining the observed side of the fork are the two immediate
arithmetic questions.

\section*{Literature and originality note}

A targeted, non-exhaustive priority search covered:

\begin{itemize}
\item circular units, the Stickelberger ideal, and the
      Main Conjecture \cite{Sinnott,Thaine,Rubin,Greither};
\item class-group computation at prime and composite cyclotomic
      conductors \cite{Schoof,Agathocleous,Miller};
\item class-rank, Gauss-sum, and Jacobi-sum constructions
      \cite{Iimura,AnglesNuccio,Gras};
\item modern computations of irregular primes and cyclotomic
      invariants \cite{Buhler}; and
\item explicit Kummer-generator work \cite{HPST}.
\end{itemize}

These sources provide close precedents for the ingredients and
philosophy.  The search did not locate the exact
conductor-\(3p\) family \eqref{eq:projected-unit}, the complete
twelve-line catalogue, or the listed split-prime certificates.
The potentially new content is that combination, together with
the exact reflected-spectrum identity.  This is evidence at the
search depth used here, not a priority clearance.  No definitive
priority claim is made; specialist review and independent
computer-algebra reproduction remain appropriate.

\section*{Acknowledgements}

Computational exploration, drafting, and verification were
assisted by OpenAI's 5.6 Sol and Anthropic's Fable 5.  Responsibility for the mathematical
statements, source selection, and final presentation remains with
the author.

\end{document}